\documentclass[preprint,11pt,nonatbib]{elsarticle}
\usepackage[utf8]{inputenc}
\usepackage[T1]{fontenc}

\usepackage{csquotes}

\makeatletter
\let\c@author\relax
\makeatother

\usepackage[backend=biber,style=apa]{biblatex}
\usepackage{amsmath}
\usepackage{float}
\usepackage{parskip}
\usepackage{subfig}
\usepackage{amsfonts}
\usepackage{amssymb}
\usepackage[ruled,vlined]{algorithm2e}
\usepackage{booktabs}
\usepackage{longtable}
\usepackage[inline]{enumitem}
\usepackage{mathtools}
\usepackage{placeins}
\usepackage[dvipsnames]{xcolor}

\usepackage{setspace}
\usepackage{geometry}
\usepackage{amsthm}
\newtheorem{proposition}{Proposition}
\newtheorem{corollary}{Corollary}
\newtheorem{example}{Example}
\theoremstyle{remark}
\newtheorem{remark}{Remark}
\theoremstyle{plain}

\usepackage{xurl}
\usepackage[pdftex,linkcolor=black,pdfborder={0 0 0}]{hyperref} 
\usepackage{calc} 
\usepackage{enumitem} 

\usepackage[all]{nowidow} 
\usepackage[protrusion=true,expansion=true]{microtype} 

\DeclareMathOperator{\minimize}{minimize}
\newcommand{\stt}{\mathrm{s.t.}}
\newcommand{\argmin}{\mathop{\mathrm{argmin\ }}\limits}

\newcommand{\F}{\mathcal{F}}
\newcommand{\E}{\mathbb{E}}

\newcommand{\R}{\mathbb{R}}
\newcommand{\T}{\mathcal{N}} 

\newcommand{\wmiddle}{\;\middle|\;} 
\newcommand{\br}[1]{\lbrack #1 \rbrack} 
\newcommand{\hist}[1]{\xi_{\br{#1}}} 
\newcommand{\histspace}[1]{\Xi_{\br{#1}}} 
\newcommand{\histpart}[2]{\Xi_{\br{#1}}^{\;#2}} 

\newcommand{\ch}{\text{ch}} 
\newcommand{\pa}{\text{pa}} 

\newcommand{\defeq}{\coloneqq}

\SetKwInput{KwIn}{Input}
\SetKwInput{KwOut}{Return}

\hypersetup{
pdfsubject = {},
pdftitle = {},
pdfauthor = {}
}

\begin{document}
\begin{frontmatter}
\title{Regret-Based Scenario-Tree Construction for Multistage
Stochastic Programs with Deterministic Constraints}

\author[1]{Robin Bruneel\corref{cor1}}
\author[1]{Mathijs Schuurmans}
\author[1]{Panagiotis Patrinos}

\affiliation[1]{organization={Department of Electrical Engineering (ESAT-STADIUS), KU Leuven},
            addressline={Kasteelpark Arenberg 10},
            city={Leuven},
            postcode={3001},
            state={Vlaams-Brabant},
            country={Belgium}}

\cortext[cor1]{Corresponding author. Email: robin.bruneel@kuleuven.be}

\begin{abstract}
Scenario trees are widely used to discretize multistage stochastic programs into forms that are tractable for numerical optimization. Yet, traditional scenario-tree construction methods are typically based on statistical similarity between the original uncertainty process and its discretization, and may fail to preserve the aspects of the uncertainty that matter most for decision-making.

This paper shows that, under a fixed probability measure and deterministic feasibility constraints, the gap between the optimal value of the original multistage stochastic program and that of its scenario-tree approximation decomposes exactly into a sum of nodewise regret terms, each measuring the optimality loss at a particular node.
This decomposition provides a clear target for scenario-tree construction: minimizing total regret across the tree. Motivated by this result, we propose a practical two-phase construction method for repeated stochastic optimization. In an offline phase, computationally inexpensive surrogates for decision regret are learned from data, and in the online phase, these surrogates guide scenario-tree construction. The resulting framework retains the structured optimization and constraint-handling capabilities of multistage stochastic programming while exploiting learned information about the uncertainty to make more effective use of a limited scenario-tree budget.

Numerical experiments on multistage stochastic storage and production-planning problems show that, for a fixed tree size, the proposed method constructs higher-quality scenario trees and achieves better closed-loop performance than traditional scenario-tree construction methods.
\end{abstract}
\begin{keyword}
Stochastic programming \sep Scenario trees \sep Scenario reduction
\sep Multistage optimization \sep Reinforcement learning
\end{keyword}
\end{frontmatter}

\section*{Nomenclature}

\noindent\textit{Abbreviations}

\noindent
\begin{tabular}{@{}p{0.2\linewidth}p{\dimexpr0.8\linewidth-2\tabcolsep\relax}@{}}
AR   & Autoregressive \\
MPC  & Model predictive control \\
DP   & Dynamic programming \\
MSP  & Multistage stochastic program \\
DQN  & Deep Q-network \\
RL   & Reinforcement learning \\
SMPC & Stochastic model predictive control \\
NAF  & Normalized advantage function
\end{tabular}

\begin{longtable*}[l]{@{}p{0.2\linewidth}p{\dimexpr0.8\linewidth-2\tabcolsep\relax}@{}}
\multicolumn{2}{@{}l}{\textit{Symbols}}\\[3pt]
\endfirsthead
\multicolumn{2}{@{}l}{\textit{Symbols (continued)}}\\[3pt]
\endhead
$t$, $T$, $t_{\mathrm{rem}}$
& Stage index, total horizon, and remaining horizon \\

$\xi_t$, $\hist{t}$, $\hist{t}^{(k)}$
& Stage uncertainty, its history, and sampled history $k$ \\

$x_t$, $x_{t,i}$
& Stage decision and shared decision at node $i$ \\

$X_t(x)$
& Feasible decision set given incoming decision $x$ \\

$\widetilde X_t(x)$
& Candidate decision set used during tree construction \\

$f_t$, $\widehat f_{t,i}$
& Stage cost and its conditional expectation at node $i$ \\

$\mathcal F_t$, $\mathcal G_t$
& Full and coarsened information at stage $t$ \\

$\T$, $\mathcal N_t$
& Scenario tree and its set of stage-$t$ nodes \\

$\kappa_t$
& Mapping from observed histories to stage-$t$ nodes \\

$\histpart{t}{i}$
& History-space cell associated with node $i$ \\

$\pa(i)$, $\ch(i)$
& Parent and children of node $i$ \\

$p_i$
& Probability of reaching node $i$ \\

$v^\star$, $\hat v(\kappa)$
& Optimal values of the original and tree problems \\

$J_\kappa(\mathbf x)$
& Expected cost of the feasible tree decisions $\mathbf x$ \\

$Q_t^\star$, $V_t^\star$
& Decision-value and optimal cost-to-go functions \\

$\widetilde Q_{t,\theta}$,
$\widetilde V_{t,\theta}$
& Surrogate decision-value and cost-to-go functions \\

$A_t^\star$, $\widetilde A_{t,\theta}$
& True and surrogate history-specific advantages \\

$\delta_{t,i}$,
$\widetilde\delta_{t,i}$
& True and surrogate nodewise regrets \\

$S_i$
& Indices of sampled scenarios assigned to node $i$ \\

$N$, $M$
& Construction sample count and evaluation episodes \\

$r_{t,i}$
& Sum of sampled surrogate advantages at stage $t$ under parent $i$ \\

$m_t$, $B$
& Additional-child budget across stage $t$ and its cap \\
\end{longtable*}

\section{Introduction}
Multistage stochastic programming provides a principled framework for sequential decision-making under uncertainty, with applications in finance \parencite{kouwenberg2001scenario}, operations \parencite{korpeoglu2011multi}, and energy systems \parencite{patrinos2011stochastic}.
Since multistage problems with continuous uncertainty are generally
intractable, finite-dimensional scenario-tree approximations
are commonly used instead. These approximations may involve two conceptually distinct operations \parencite{carpentier2009approximations}. 
The first groups observation histories into a finite number of tree nodes, requiring histories within each node to share a common decision and thereby coarsening the information available for decision-making.
The second replaces grouped uncertainty realizations by node representatives, yielding a discrete approximation of the underlying probability distribution.
In this paper, we isolate the effect of information coarsening on decision quality while keeping the underlying probability measure fixed.

Popular scenario-tree construction methods include sampling-based approaches \parencite{dupacova2000scenarios,infanger1992monte}, moment-matching techniques \parencite{hoyland2001momentmatching}, and optimal discretization methods based on probability metrics such as Wasserstein and nested distances \parencite{pflug2012distance,heitsch2009multistage,bruneel2025multistage}. These methods primarily target statistical fidelity of the uncertainty representation, rather than accounting directly for how grouping histories affects the optimization problem. Under a limited node budget, a statistically
accurate tree may therefore fail to preserve the distinctions that
matter most for decision-making.

Problem-dependent construction methods address this limitation by incorporating decision information into scenario selection and aggregation. In the two-stage setting, several authors have proposed methods that account for the impact of scenarios on the optimal solution \parencite{BertsimasDimitris2023OSRf,hewitt2022decisionbased,henrion2022problembased,KeutchayanJulien2023Psci}.
Extending these methods to multiple stages requires assessing
how the information coarsening affects both current decisions and
future recourse.
\textcite{keutchayan2020multistage} derive a decomposition involving optimization and numerical discretization
errors and guide construction using an integration-error upper bound.
\textcite{dempster2006sequential} adapts branching using
expected-value-of-perfect-information criteria obtained from
solutions of successive sampled optimization problems.

Our approach guides tree construction using an exact decomposition
of the information-coarsening loss into nodewise regret terms,
rather than an integration-error upper bound.
We focus on settings where closely related
multistage optimization problems are solved repeatedly as new
information becomes available. This repeated use allows the
upfront cost of learning decision-value surrogates to be
amortized over many subsequent optimizations. Once trained,
these surrogates provide inexpensive regret estimates that
guide both the grouping of histories and the allocation of
branches, avoiding repeated solution of the multistage
problem during tree construction.

This paper makes two methodological contributions:
\begin{itemize}
    \item We derive a nodewise regret representation of the
    information-coarsening gap under a fixed probability measure
    and deterministic feasibility constraints. This exact
    representation provides a problem-dependent objective for
    scenario-tree construction.

    \item We develop a greedy stagewise algorithm that constructs
scenario trees by approximately minimizing regret estimated
from decision-value surrogates learned offline. A surrogate-error
bound relates the construction criterion to the true tree
optimality gap.
\end{itemize}

Experiments on stochastic energy storage and production
planning demonstrate improved closed-loop performance over
the considered sampling- and distance-based construction
methods at matched tree sizes. The gains persist even
when the surrogate's standalone greedy policy performs poorly.

The remainder of this paper is organized as follows.
Section~\ref{sec:problem_formulation} introduces the multistage stochastic optimization problem.
Section~\ref{sec:clustering} derives the information-coarsening regret decomposition, and Section~\ref{sec:tree_construction} develops the construction method.
Section~\ref{sec:results} presents the numerical experiments.

\section{Problem Formulation}\label{sec:problem_formulation}
We consider a discrete-time, finite-horizon multistage stochastic optimization problem in which decisions are taken sequentially as uncertainty unfolds over time.
Throughout the paper, we assume deterministic feasibility constraints, so that the uncertainty affects only the objective function.
Let $(\Omega, \mathcal{F}, \mathbb{P})$ be a probability space and let $T \in \mathbb{N}$ denote the planning horizon.
For $n\in\mathbb N$, write $\br{n}\defeq\{1,\ldots,n\}$ if $n > 0$, (and $\br{0} = \emptyset$).
At each stage $t \in \br{T}$, the newly observed uncertainty is represented by a random variable $\xi_t: \Omega \to \Xi_t \subseteq \mathbb{R}^{n_\xi}$ and we denote the associated data process by
\begin{equation}
    \hist{t} \defeq (\xi_1, \xi_2, \ldots, \xi_t): \Omega \to \histspace{t}, \qquad \histspace{t} \defeq \prod_{s=1}^t \Xi_s.
\end{equation}
The information available for decision-making is described by the natural filtration $(\mathcal{F}_t)_{t=0}^T$, where $\mathcal{F}_t$ is the $\sigma$-algebra generated by the
process $\hist{t}$:
\begin{equation}
    \mathcal{F}_t \defeq \sigma(\hist{t}), \quad t \in \br{T}, \quad \mathcal{F}_0 = \{\emptyset, \Omega\}.
\end{equation}

At stage $t$, the decision-maker observes $\xi_t$ and uses this information to select
a decision $x_t:\Omega\to\mathcal{X}_t\subseteq\mathbb{R}^{n_x}$.
To ensure non-anticipativity, each $x_t$ must be
$\mathcal{F}_t$-measurable. Interstage constraints further restrict
this choice to $X_t(x_{t-1})$, where
$X_t:\mathcal{X}_{t-1}\rightrightarrows\mathcal{X}_t$ is a
deterministic set-valued mapping. Each feasible decision incurs
a stage cost $f_t(x_t,\xi_t)$, where
$f_t:\mathcal{X}_t\times\Xi_t\to\mathbb{R}$ is measurable.
We seek a decision process $(x_t)_{t=1}^T$ that minimizes the
expected sum of these costs while satisfying the above requirements.
For a fixed initial decision $x_0\in\mathcal{X}_0$, this leads
to the following multistage stochastic program (MSP):
\begin{equation}\label{eq:full-problem}
v^\star \defeq \inf_{(x_t)_{t=1}^T}\left\{
\E\!\left[\sum_{t=1}^T f_t(x_t,\xi_t)\right]
\ \wmiddle\
\begin{array}{ll}
x_t\in X_t\!\left(x_{t-1}\right) \text{ a.s.}, & t\in\br{T},\\[1mm]
x_t \text{ is }\mathcal F_t\text{-measurable}, & t\in\br{T}
\end{array}
\right\}.
\end{equation}
Under standard measurability and integrability assumptions, the multistage problem also admits a dynamic programming (DP) formulation. For an incoming decision $x \in \mathcal{X}_{t-1}$ and history $\hist{t} \in \histspace{t}$, let $V_t^\star(x,\hist{t})$ denote the optimal expected cost from stage $t$ through stage $T$.
The cost-to-go functions $V_t^\star : \mathcal{X}_{t-1} \times \histspace{t} \to \mathbb{R}$ satisfy
\[
V_t^\star(x,\hist{t})
= \inf_{x^+\in X_t(x)}
\Big\{
f_t(x^+,\xi_{t})+\mathbb E\left[V_{t+1}^\star(x^+,\hist{t+1})\wmiddle \hist{t}\right]
\Big\},
\]
where $x^+$ denotes the current decision and
$V_{T+1}^\star\equiv0$.
The optimal value before the first observation is
\[
    v^\star=\E\!\left[V_1^\star(x_0,\hist{1})\right].
\]

The multistage stochastic program \eqref{eq:full-problem} is closely related to a class of stochastic optimal control problems 
\parencite{lan2024numerical,bertsekas2004stochastic}, with the caveat that 
due to the deterministic nature of the inter-stage constraint mapping $X_t$,
stochastic dynamics are generally not admissible.

\begin{example}[Optimal control]\label{ex:optimal-control}
  Consider the discrete-time dynamical system 
  \(
    z_{t+1} = \varphi_t(z_t, u_t),
  \)
  with states $z_t \in \R^{n_z}$, $t \in [T+1]$, and control
  inputs $u_t \in \R^{n_u}$, $t \in [T]$.
  The associated stochastic optimal control problem is
  \begin{equation}\label{eq:stochastic-ocp}
    \begin{aligned}
      \minimize\;&
      \E\!\left[\sum_{t=1}^{T} \ell_t(z_t, u_t, \xi_t) + \ell_{T+1}(z_{T+1})\right]\\
      \stt\;&
      z_{t+1} = \varphi_t(z_t, u_t), \text{ a.s.}, & t\in [T],\\
                                                   &(z_t, u_t) \in Z_t, \text{ a.s.}, & t \in [T],\\
                                                   &u_t \text{ is }\mathcal F_t\text{-measurable}, & t\in [T],\\
                                                   &z_{T+1} \in Z_{T+1}, \text{ a.s.},\\
    \end{aligned}
  \end{equation}
  where $z_1=\bar z$ is a given initial state,
  and $Z_t \subseteq \R^{n_z + n_u}$, $t \in [T]$,
  and $Z_{T+1} \subseteq \R^{n_z}$ are given constraint sets.
  Assume that the dynamics maps and cost functions are measurable,
  the constraint sets are Borel, and the costs satisfy the standing
  integrability assumptions.
  By induction, using measurability of the dynamics,
  every adapted control process induces states
  $z_t$ that are $\mathcal F_{t-1}$-measurable, so that
  $x_t\defeq(z_t,u_t)$ is $\mathcal F_t$-measurable for every
  $t\in\br{T}$.

  For ease of notation, define the constant map
  $\varphi_0(x) \equiv \bar z$.
  Then, for any given $x_0 \in \R^{n_z + n_u}$, define
  \[
    X_t(x_{t-1}) = \big\{ x = (z, u) \in Z_t \mid z = \varphi_{t-1}(x_{t-1}) \big\}, \quad t \in [T-1],
  \]
  and $X_T(x_{T-1}) = \{ (z, u) \in Z_T \mid z = \varphi_{T-1}(x_{T-1}), \; \varphi_T(z,u) \in Z_{T+1}\}$.
  Setting 
  \[
    f_t((z,u),\xi)\defeq
    \begin{cases}
        \ell_t(z,u,\xi), &1\leq t<T,\\
        \ell_T(z,u,\xi)
        +\ell_{T+1}\!\left(\varphi_T(z,u)\right), &t=T,
    \end{cases}
  \]
  we obtain a correspondence between feasible adapted processes in \eqref{eq:stochastic-ocp} and
  \eqref{eq:full-problem}, with identical total costs almost surely.
  Hence, these problems are equivalent.
\end{example}

In the absence of state constraints, \eqref{eq:full-problem} can 
additionally model stochastic optimal control problems for linear uncertain systems, 
using a change of variables akin to \textit{purified outputs} \parencite[\S 4.2]{ben2006extending}.

\begin{example}[Stochastic linear optimal control]
  Consider the stochastic optimal control problem
  \begin{equation}\label{eq:stochastic-ocp-lin}
    \begin{aligned}
      \minimize\;&
      \E\!\left[\sum_{t=1}^{T} \ell_t(z_t, u_t, \xi_t) + \ell_{T+1}(z_{T+1})\right]\\
      \stt\;&
      z_{t+1} = A_t z_t + B_t u_t + G_t \xi_t, \text{ a.s.}, & t\in [T],\\
                                                   &u_t \in U_t, \text{ a.s.}, & t \in [T],\\
                                                   &u_t \text{ is }\mathcal F_t\text{-measurable}, & t\in [T],
    \end{aligned}
  \end{equation}
  where for all $t \in [T]$, $U_t \subseteq \R^{n_u}$ is a deterministic constraint set.
  Introduce the change of variables
  \(
    e_1 = 0,
  \)
  \( 
    e_{t+1} = A_t e_t + G_t \xi_t
  \),
  \(
    q_t = z_t - e_t,
  \)
  the dynamics recursion becomes 
  \begin{equation*}
      q_{t+1} = A_t q_t + B_t u_t, 
  \end{equation*}
  with $q_1 = z_1$. Defining $\eta_t = (\xi_t, e_t)$, it is easy to see
  that $\eta_t$ is $\F_t$-measurable for all $t \in [T]$ (in fact,
  $\sigma(\eta_{[t]}) = \sigma(\xi_{[t]}) = \F_t$).
  Setting $x_t = (q_t, u_t)$, and defining 
  \(
    f_t((q, u), (\xi, e)) = \ell_t(q + e, u, \xi),
  \) 
  \(t \in [T-1] \), and 
  \[ 
    f_T((q, u), (\xi, e)) =
      \ell_T(q + e, u, \xi) + \ell_{T+1}(A_T(q+e) + B_T u + G_T \xi),
  \]
  we arrive at a special case of \eqref{eq:stochastic-ocp}.
\end{example}

Nonlinear stochastic control problems with deterministic input constraints and no state constraints also fit \eqref{eq:full-problem}.
The states can be eliminated by expressing them in terms of control and disturbance histories, which then serve as the decision and uncertainty variables, respectively. In that setting, the dimensions of both the decision and uncertainty variables grow with the horizon.


\section{Scenario Tree Approximation}\label{sec:clustering}
In this section, we construct an approximation of \eqref{eq:full-problem} by coarsening the information structure available to the decision maker.
Specifically, we group observation histories into finitely
many nodes at each stage and require histories within
each node to share a common decision. The result is a finite-dimensional and computationally more tractable tree problem, at the expense of a potential loss in optimality.
We first introduce a scenario-tree notation to represent such coarsenings, then derive the corresponding optimization problem, and analyze its optimality loss.

\subsection{Scenario Trees}\label{sec:scenario_trees}
We represent the coarsened filtration by a rooted scenario tree $\T$, where $\mathcal N_t \subseteq \T$ denotes the set of nodes at stage $t$ so that $\T = \bigcup_{t=0}^T \mathcal{N}_t$.
Each node $i \in \mathcal{N}_t$ has a unique parent $\pa(i) \in \mathcal{N}_{t-1}$.
For every nonterminal node $i \in \mathcal{N}_t$, its children are given by
\[
\ch(i) \defeq \left\{\, j\in\mathcal N_{t+1} \wmiddle \pa(j) = i \,\right\}.
\]
Nodes in $\mathcal{N}_T$ are terminal and therefore have no children.
The coarsened information structure is represented by a family of mappings
\[
\kappa_t : \histspace{t} \to \mathcal N_t, \qquad t \in \br{T},
\]
which assigns every scenario history $\hist{t}$ to a node in the scenario tree.
The mapping $\kappa_t$ induces a partition of the history space $\histspace{t}$ into disjoint cells. For each node $i \in \mathcal{N}_t$, its associated partition cell is:
\begin{equation}\label{eq:def-part}
    \histpart{t}{i} \defeq \kappa_t^{-1}(\{i\}) = \left\{\hist{t} \in \histspace{t} \wmiddle \kappa_t(\hist{t}) = i\right\},
\end{equation}
where every node $i\in\mathcal N_t$ has probability
$p_i=\mathbb P(\hist{t}\in\histpart{t}{i})$,
with $\sum_{i\in\mathcal N_t}p_i=1$.
An example scenario tree is shown in Figure~\ref{fig:scenario_tree}.
The assignments must be consistent with the tree's parent relation:
\begin{equation}\label{eq:time-consistency}
    \kappa_{t-1}(\hist{t-1})
    =
    \pa\!\left(\kappa_t(\hist{t})\right),
    \qquad t=2,\ldots,T.
\end{equation}
Thus, histories assigned to the same node must also be assigned to
the same parent node. This time consistency ensures that the node
index at stage $t$ determines the preceding node indices and that
the resulting information structure is a filtration.

\begin{figure}
    \centering
    \includegraphics[width=0.3\textwidth]{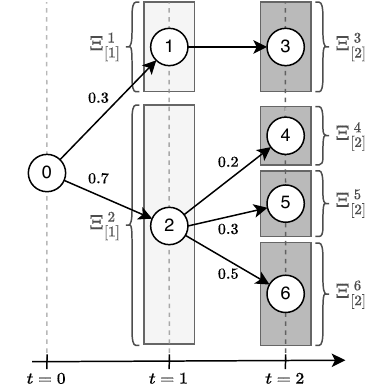}
    \caption{A scenario tree representing time-consistent partitions
    of uncertainty histories. Each node corresponds to a cell of the
    history space at its stage, and the edge labels indicate
    conditional transition probabilities.}\label{fig:scenario_tree}
\end{figure}

\subsection{Scenario-Tree Problem Formulation}
We approximate the original multistage stochastic problem by replacing its information structure with a discrete scenario tree.
Under a time-consistent stagewise partitioning $\kappa = (\kappa_1,\ldots,\kappa_T)$, the information available at stage $t$ is reduced from the full history $\hist{t}$ to the node index $\kappa_t(\hist{t})$.
This approximation leads to a coarser information structure $(\mathcal{G}_t)_{t=0}^T$ defined by
\begin{equation}
    \mathcal G_0=\mathcal F_0,\qquad
\mathcal G_t\defeq\sigma(\kappa_t\circ\hist{t})
\subseteq\mathcal F_t,\quad t\in\br{T}.
\end{equation}
The resulting information-coarsened problem is obtained by restricting admissible decision processes to be adapted to this coarser filtration:
\begin{equation}\label{eq:discretized-problem-og}
    \hat{v}(\kappa) \defeq \inf_{(x_t)_{t=1}^T} \left\{
    \mathbb{E}\left[ \sum_{t=1}^T f_t(x_t, \xi_t) \right]
    \wmiddle
    \begin{array}{ll}
    x_t \in X_t\left(x_{t-1}\right) \text{ a.s.}, & t\in \br{T},\\
    x_t \text{ is } \mathcal{G}_t\text{-measurable}, & t\in \br{T}
    \end{array}
    \right\}.
\end{equation}
The notation $\hat{v}(\kappa)$ emphasizes the dependence on the chosen partitioning.
Because $\mathcal G_t \subseteq \mathcal F_t$, every feasible policy in \eqref{eq:discretized-problem-og} is also feasible in \eqref{eq:full-problem}, and hence $\hat v(\kappa)\ge v^\star$.
Moreover, since all histories assigned to the same node share the same decision, \eqref{eq:discretized-problem-og} admits the equivalent finite-dimensional tree formulation stated in Proposition~\ref{lem:discretized_DP}.

\begin{proposition}[Representation of the information-coarsened problem]\label{lem:discretized_DP}
The information-coarsened problem~\eqref{eq:discretized-problem-og} is equivalent to a finite-dimensional tree problem with
node decisions $\mathbf x = \{x_{t,i}\}_{t\in\br{T},\,i\in\mathcal N_t}$:
\begin{equation}\label{eq:discretized-problem}
    \hat v(\kappa) = \inf_{\{x_{t,i}\}}
    \left\{
    \sum_{t=1}^T \sum_{i\in\mathcal N_t} p_{i}\,\widehat f_{t,i}(x_{t,i})
    \wmiddle
    x_{t,i}\in X_{t}\left(x_{t-1,\pa(i)}\right)
    \right\},
\end{equation}
where
\begin{align*}
    \widehat{f}_{t,i}(x) \defeq & \ \mathbb{E}\left[ f_t(x, \xi_{t}) \wmiddle \hist{t} \in \histpart{t}{i} \right].
\end{align*}
\end{proposition}

\begin{proof}
Let $(x_t)_{t=1}^T$ be feasible for \eqref{eq:discretized-problem-og}.
Since each $x_t$ is $\mathcal G_t$-measurable and $\mathcal G_t$ is generated by the mapping $\kappa_t\circ\hist{t}$, the decision $x_t$ must be constant
on each cell $\histpart{t}{i}$. That is, there exist $x_{t,i}\in\mathcal X_t$ such that
\[
x_t(\omega)=x_{t,i} \quad \forall \omega \in \Omega: \hist{t}(\omega)\in\histpart{t}{i}, \quad \forall i \in \mathcal{N}_t.
\]
Finally, since $\big\{ \histpart{t}{i} \big\}_{i \in \mathcal{N}_t}$ is a partition of $\histspace{t}$, the expected stage cost decomposes into:
\begin{align*}
    \mathbb E\!\left[f_t(x_t,\xi_{t})\right]
    =&\sum_{i\in\mathcal N_t} \mathbb P(\hist{t} \in\histpart{t}{i})\
    \mathbb E\!\left[f_t(x_{t,i},\xi_{t})\mid \hist{t}\in\histpart{t}{i}\right]\\
    = &\sum_{i\in\mathcal N_t} p_{i}\,\widehat f_{t,i}(x_{t,i}),
\end{align*}
which yields the objective of \eqref{eq:discretized-problem}.
Conversely, any feasible collection $\{x_{t,i}\}$ for \eqref{eq:discretized-problem} defines a feasible
$\mathcal G_t$-measurable policy via $x_t(\omega)\defeq x_{t,\kappa_t(\hist{t}(\omega))}$. This establishes the equivalence and completes the proof.
\end{proof}

\subsection{Regret Decomposition of the Optimality Gap}
In this section, we characterize the optimality gap between the original problem
\eqref{eq:full-problem} and its scenario-tree approximation
\eqref{eq:discretized-problem}.
Throughout this section, we assume that the value functions
are finite and that the cost and value-function terms are
measurable and integrable along the feasible policies under consideration.
For a feasible decision $x^+ \in \mathcal{X}_t$ and observed history $\hist{t} \in \histspace{t}$, we
define the decision-value function of the original problem as
\begin{align}\label{eq:q-function}
    Q_t^\star(x^+, \hist{t}) \defeq & f_{t}(x^+,\xi_{t}) + \mathbb{E}\left[ V_{t+1}^\star(x^+, \hist{t+1}) \wmiddle \hist{t} \right],
\end{align}
so that $V_t^\star(x, \hist{t}) = \inf_{x^+\in X_t(x)} Q_t^\star(x^+,\hist{t})$. This function captures the expected cost of taking a specific decision $x^+$ at stage $t$ given the observed history $\hist{t}$ and then following the optimal policy thereafter.
In the control interpretation of the multistage stochastic program,
this corresponds to the optimal action-value function, or $Q$-function,
used in reinforcement learning \parencite{sutton1998reinforcement},
expressed here under a cost-minimization convention.

For an incoming decision $x$ and a feasible current decision $x^+\in X_t(x)$,
define the advantage function
\begin{equation}
A_t^\star(x^+;x,\hist{t})
\defeq Q_t^\star(x^+,\hist{t})-V_t^\star(x,\hist{t}).
\label{eq:advantage-function}
\end{equation}
Under the cost-minimization convention, this is the excess cost of choosing
$x^+$ instead of an optimal decision for the observed history.

For any node $i\in\mathcal N_t$, incoming decision $x\in\mathcal X_{t-1}$, and candidate decision
$x^+\in X_t(x)$, the nodewise regret is the conditional mean of this advantage:
\begin{equation}\label{eq:nodewise-regret}
\delta_{t,i}(x^+;x)
\defeq
\mathbb E\!\left[
A_t^\star(x^+;x,\hist{t})
\ \middle|\
\hist{t}\in \histpart{t}{i}
\right].
\end{equation}
This quantity measures the expected cost increase of enforcing the common decision $x^+$
across all scenario histories assigned to node $i$,
instead of the optimal history-dependent decision.
In this comparison, only the current decision is forced to be common across
the node. Subsequent decisions are assumed to use the full observation
history optimally.
Note that $\delta_{t,i}(x^+;x) \ge 0$, with equality if and only if $x^+$ is an
optimal decision for almost all histories in $\histpart{t}{i}$.
The following proposition shows that these local penalties, weighted by node
probabilities, add up to the total loss of the tree policy.

\begin{proposition}[Regret decomposition]
\label{lem:error_decomp}
Let $\mathbf x=\{x_{t,i}\}$ be any feasible collection of
node decisions for a measurable, time-consistent partitioning
$\kappa$. Define the induced decision
policy by $x_t(\omega)=x_{t,\kappa_t(\hist{t}(\omega))}$ and
its expected cost by
\[J_\kappa(\mathbf x)\defeq\mathbb E\left[\sum_{t=1}^T f_t(x_t,\xi_t)\right].\]
Then
\begin{equation}
J_\kappa(\mathbf x)-v^\star
=
\sum_{t=1}^T\sum_{i\in\mathcal N_t}
p_i\,\delta_{t,i}
\bigl(x_{t,i};x_{t-1,\pa(i)}\bigr).
\label{eq:nodewise-regret-decomposition}
\end{equation}
In particular, for an optimal tree solution
$\hat{\mathbf x}$ of \eqref{eq:discretized-problem}, the left-hand side equals
$\hat v(\kappa)-v^\star$.
\end{proposition}

\begin{proof}
By time consistency, the incoming decision at node $i$ is
$x_{t-1,\pa(i)}$. Using \eqref{eq:nodewise-regret} and \eqref{eq:advantage-function}, we obtain
\begin{align*}
&\sum_{t=1}^T\sum_{i\in\mathcal N_t}
p_i\,\delta_{t,i}
\bigl(x_{t,i};x_{t-1,\pa(i)}\bigr)
=
\mathbb E\!\left[
\sum_{t=1}^T
\bigl(Q_t^\star(x_t,\hist{t})
      -V_t^\star(x_{t-1},\hist{t})\bigr)
\right].
\end{align*}
Substituting the definition of $Q_t^\star$ and cancelling successive value-function terms
yields
\[
\mathbb E\!\left[\sum_{t=1}^T f_t(x_t,\xi_t)\right]
-\mathbb E[V_1^\star(x_0,\hist{1})]
=
J_\kappa(\mathbf x)-v^\star.
\]
The optimal-tree statement follows from
$J_\kappa(\hat{\mathbf x})=\hat v(\kappa)$.
\end{proof}

The decomposition expresses the finite-horizon performance-difference
identity~\parencite[Lemma~6.1]{kakade2002approximately} in scenario-tree form.
Evaluating it at the
optimal node decisions gives the ideal tree-construction
criterion, with the dependence of the probabilities and
decisions on the partitioning made explicit:
\begin{equation}
\min_{\kappa\in\mathcal K} \hat v(\kappa)-v^\star
=
\min_{\kappa\in\mathcal K}
\sum_{t=1}^T\sum_{i\in\mathcal N_t}
p_i(\kappa)\,
\delta_{t,i}\bigl(
\hat x_{t,i}(\kappa);\hat x_{t-1,\pa(i)}(\kappa), \kappa
\bigr),
\label{eq:ideal-regret-tree-objective}
\end{equation}
where $\mathcal K$ contains all measurable, time-consistent partitionings
within the node budget.

\begin{remark}[Connection to state abstraction]
In the control formulation, information coarsening maps the state
$\big(x_{t-1},\hist{t}\big)$ to $\big(x_{t-1},\kappa_t(\hist{t})\big)$,
retaining the incoming decision while replacing the uncertainty
history by its scenario-tree node index.
This is a form of state abstraction in the terminology of Markov
decision processes \parencite[\S3.1]{li2006towards}.
\end{remark}

\begin{example}\label{ex:regret}
Consider a two-stage problem with $\xi_1=0$ and equally likely
second-stage realizations $\xi_2\in\{-10,-8,-2,4\}$.
The decisions satisfy $x_1\in[0,4]$ and $x_2\in[-x_1,x_1]$,
with stage costs $\tfrac12 x_1^2$ and $\xi_2x_2$, respectively.
After observing $\xi_2$, the optimal second-stage decision is
$x_2=-x_1\operatorname{sgn}(\xi_2)$, giving expected second-stage
cost $-x_1\mathbb E[|\xi_2|]=-6x_1$. Hence,
\[
v^\star=\min_{0\le x_1\le4}
\left\{\tfrac12x_1^2-6x_1\right\}=-16,
\]
attained at $x_1=4$.
Figure~\ref{fig:regret-example}a
shows the four realizations and their optimal decisions.
We compare two ways of grouping these realizations into
two second-stage nodes.

The partition $\{-10,-8\}$ and $\{-2,4\}$, shown in
Figure~\ref{fig:regret-example}b,
has conditional means $-9$ and $1$, each with probability $1/2$.
Among all distributions supported on at most two points,
the resulting representative distribution minimizes the
2-Wasserstein distance to the original second-stage distribution.
Its squared distance is $5$.
However, this grouping incurs positive regret,
as we show below.
The two nodes have probability $1/2$ and conditional means
$-9$ and $1$, respectively.
For fixed $x_1$, their optimal shared decisions are therefore
$x_{2,2}=x_1$ and $x_{2,3}=-x_1$, giving
\[
\hat v(\kappa)=\min_{0\le x_1\le4}
\left\{\tfrac12x_1^2-5x_1\right\}=-12.
\]
The optimal tree decisions are $\hat x_{1,1}=4$,
$\hat x_{2,2}=4$, and $\hat x_{2,3}=-4$.

The first-stage decision remains optimal for the original
problem, and the shared decision at node $2$ is optimal for
both of its realizations. Their regret terms are therefore zero.
At node $3$, the shared decision $-4$ is optimal for $\xi_2=4$,
but not for $\xi_2=-2$: in the latter case, choosing $-4$
instead of $4$ raises the second-stage cost from $-8$ to $8$.
Averaging this loss within the node gives
$\delta_{2,3}(-4;4)=\tfrac12(16+0)=8$.
The regret decomposition therefore yields
\[
\hat v(\kappa)-v^\star
=(-12)-(-16)
=0+\tfrac12\cdot0+\tfrac12\cdot8=4.
\]

The alternative partition $\{-10,-8,-2\}$ and $\{4\}$,
shown in
Figure~\ref{fig:regret-example}c,
groups realizations by sign.
Its nodes have probabilities $3/4$ and $1/4$.
For every fixed $x_1$, each group shares an optimal second-stage
decision, so this partition preserves the original objective
and attains $v^\star$.
Its regret is therefore zero, despite its larger squared
2-Wasserstein distance of $26/3$.
Thus, a Wasserstein-optimal approximation can lead to positive
regret even when another tree with the same node budget
preserves the original optimum.
\end{example}

\begin{figure}[tbp]
    \centering
    \subfloat[Full information\label{fig:regret-example-original}]{%
        \includegraphics[height=0.30\textwidth]{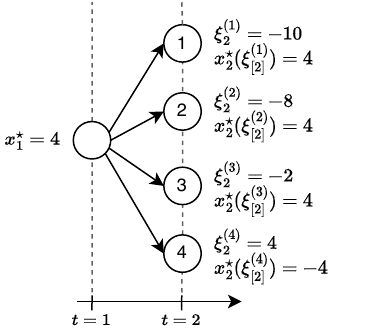}}
    \hfill
    \subfloat[Distance-based grouping\label{fig:regret-example-pairs}]{%
        \includegraphics[height=0.30\textwidth]{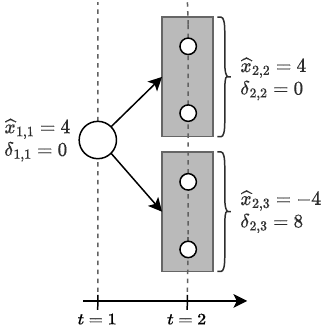}}
    \hfill
    \subfloat[Regret-based grouping\label{fig:regret-example-signs}]{%
        \includegraphics[height=0.30\textwidth]{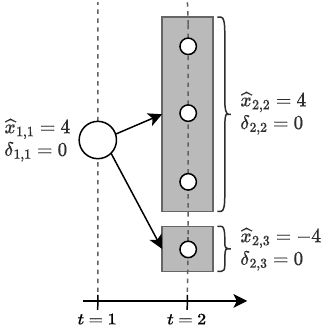}}
    \caption{Information coarsening in Example~\ref{ex:regret}.
    The four second-stage realizations are equally likely and appear in the
    same order in all panels. Shaded regions group realizations that share
    a decision. Panel (a): Full information yields $v^\star=-16$.
    Panel (b): The distance-minimizing grouping $\{-10,-8\}$ and $\{-2,4\}$ gives $\hat v(\kappa)=-12$:
    node $3$ has conditional regret $8$ and probability $1/2$, so the
    optimality gap is $4$.
    Panel (c): Grouping $\{-10,-8,-2\}$ and $\{4\}$ preserves the optimal
    decisions and yields zero regret.}
    \label{fig:regret-example}
\end{figure}

\section{Regret-Based Scenario Tree Construction}\label{sec:tree_construction}
This section develops a practical tree-construction algorithm
based on the regret decomposition of Proposition~\ref{lem:error_decomp}.
We first introduce learned decision-value functions to estimate
regret, then use these estimates to guide stagewise refinement of the tree.
During construction, temporary decisions determine the incoming
conditions used to build subsequent stages. We call these decisions
auxiliary because, once the tree is complete, its stochastic program
optimizes all decisions again using the original costs and constraints.

\subsection{Surrogate Decision-Value Functions}
Computing the true decision-value functions is generally as
difficult as solving the underlying multistage problem. To make
this criterion practical, we use parametric surrogates
$\widetilde Q_{t,\theta}$ learned offline, where $\theta$ denotes
the learned parameters.
For each incoming decision $x$, let
$\widetilde X_t(x)\subseteq X_t(x)$ be a nonempty set of candidate
decisions on which the surrogate is defined. This may be the full
feasible set or a restricted subset, such as a finite decision set.
Minimizing over the candidate decisions gives the surrogate
cost-to-go function
\begin{equation}\label{eq:surrogate_value}
    \widetilde V_{t,\theta}(x,\hist{t})
    \defeq
    \inf_{x^+\in\widetilde X_t(x)}
    \widetilde Q_{t,\theta}(x^+,\hist{t}).
\end{equation}
For a candidate decision $x^+\in\widetilde X_t(x)$, the surrogate
counterpart of the advantage function \eqref{eq:advantage-function} is
\begin{equation}
\widetilde A_{t,\theta}(x^+;x,\hist{t})
\defeq \widetilde Q_{t,\theta}(x^+,\hist{t})
-\widetilde V_{t,\theta}(x,\hist{t}).
\label{eq:advantage-functions}
\end{equation}
It compares the learned value of $x^+$ with the infimum over
candidate decisions at the same history.
This baseline may exclude better decisions available in the
original feasible set $X_t(x)$.
The nodewise surrogate regret is the conditional mean of this
advantage over the histories assigned to node $i$:
\begin{equation}\label{eq:surrogate_regret}
    \widetilde\delta_{t,i}(x^+;x)
    \defeq
    \E\!\left[
        \widetilde A_{t,\theta}(x^+;x,\hist{t})
        \wmiddle
        \hist{t}\in\histpart{t}{i}
    \right].
\end{equation}

\paragraph{Effect of surrogate error}
The following bound connects total expected surrogate regret
to the true information-coarsening loss.
It shows that the surrogate construction criterion, together
with an allowance for advantage-approximation error, bounds
the optimality gap after optimizing the completed tree.

\begin{corollary}[Optimality-gap bound under regret approximation]
\label{cor:surrogate-regret-bound}
Consider the setting of Proposition~\ref{lem:error_decomp}, with all its assumptions in force.
Assume that all retained nodes have positive probability and that
the true and surrogate value functions are finite and measurable,
with integrable advantages. Suppose that, for a finite constant
$\varepsilon\geq0$ and every $t\in\br{T}$,
\begin{equation}
\sup_{\substack{x\in\mathcal X_{t-1}\\
                x^+\in\widetilde X_t(x)}}
\left(
A_t^\star(x^+;x,\hist{t})
-\widetilde A_{t,\theta}(x^+;x,\hist{t})
\right)
\leq\varepsilon
\label{eq:surrogate-regret-error}
\end{equation}
holds almost surely.
Then, for every feasible collection of auxiliary node decisions
$\mathbf x=\{x_{t,i}\}$ with
$x_{t,i}\in\widetilde X_t(x_{t-1,\pa(i)})$,
the optimality gap after optimizing the tree decisions over the
original feasible sets satisfies
\begin{equation}
0\leq\hat v(\kappa)-v^\star
\leq\sum_{t=1}^T\sum_{i\in\mathcal N_t}
p_i\,
\widetilde\delta_{t,i}
\bigl(x_{t,i};x_{t-1,\pa(i)}\bigr)
+T\varepsilon.
\label{eq:surrogate-tree-bound}
\end{equation}
\end{corollary}

\begin{proof}
Fix any such auxiliary decisions. At each node $i\in\mathcal N_t$,
write $x=x_{t-1,\pa(i)}$. Evaluating
\eqref{eq:surrogate-regret-error} at $x^+=x_{t,i}$ and taking
conditional expectations over that node gives
$\delta_{t,i}(x_{t,i};x)\leq
\widetilde\delta_{t,i}(x_{t,i};x)+\varepsilon$.
Applying Proposition~\ref{lem:error_decomp} to the auxiliary
decisions and using $\sum_{i\in\mathcal N_t}p_i=1$ at every
stage yields
\[
J_\kappa(\mathbf x)-v^\star
\leq
\sum_{t=1}^T\sum_{i\in\mathcal N_t}
p_i\,\widetilde\delta_{t,i}
\bigl(x_{t,i};x_{t-1,\pa(i)}\bigr)+T\varepsilon.
\]
Since $\mathbf x$ is feasible for the final tree problem,
$v^\star\leq\hat v(\kappa)\leq J_\kappa(\mathbf x)$,
which proves the result.
\end{proof}

This bound applies uniformly across admissible partitions and
feasible auxiliary decisions, which need not be optimal for the
completed tree.
It motivates reducing total surrogate regret during construction:
when this term is large relative to $T\varepsilon$, it dominates
the upper bound and offers the greatest scope for tightening it.

The uniform error bound may be conservative because it also
covers decisions never used during construction. Restricting
the error bound to the auxiliary decisions actually used, over
their relevant histories, could yield a tighter guarantee.
In either formulation, the error bound is an assumption that
the training procedure does not itself certify.

\begin{remark}[Comparison with the surrogate-greedy policy]
Let $\widetilde\pi$ denote the surrogate-greedy policy, which
selects the unique minimizer of $\widetilde Q_{t,\theta}$
over $\widetilde X_t(x)$ almost surely at each stage.
If a scenario tree
admits feasible auxiliary decisions with zero total
expected surrogate regret, those decisions agree with
the surrogate-greedy policy almost surely. The greedy
policy can therefore be represented on the tree, and
optimizing the tree decisions cannot yield a higher
expected cost:
\[
\hat v(\kappa)\leq J(\widetilde\pi),
\]
where $J(\widetilde\pi)$ denotes the expected cost of
$\widetilde\pi$ under the same initial conditions and
probability measure.
This gives a sufficient condition for the optimized tree
to match or outperform the greedy policy in expected cost.
\end{remark}

\subsection{Scenario Tree Construction}
With the surrogate regret function defined, two difficulties remain: (i) evaluating the conditional expectations in~\eqref{eq:surrogate_regret}, and (ii) solving the resulting tree-design problem. Our practical method addresses these difficulties by
(i) approximating the conditional expectations using Monte Carlo averages
and (ii) constructing the tree through greedy stagewise refinement.

\paragraph{Monte-Carlo approximation}
To make the conditional expectations in~\eqref{eq:surrogate_regret} computationally tractable, we approximate the sample space using $\{ \omega_k \}_{k=1}^N \subset \Omega$ drawn independently according to $\mathbb P$.
Write $\hist{t}^{(k)}\defeq\hist{t}(\omega_k)$ for the
history of sample $k$.
For each node $i \in \mathcal{N}_t$, let
\[
S_i \defeq \{\, k \in \br{N} \mid \kappa_t(\hist{t}^{(k)}) = i \,\}
\]
denote the indices of its assigned samples. For
$x^+\in\widetilde X_t(x)$, nodewise surrogate regret is
estimated by averaging the surrogate advantage:
\begin{align*}
    \widetilde{\delta}_{t,i}(x^+;x)
    \approx
    \frac{1}{|S_i|}
    \sum_{k\in S_i}
    \widetilde A_{t,\theta}(x^+;x,\hist{t}^{(k)}).
\end{align*}
The empirical node probability is $p_i=|S_i|/N$, and conditional
stage costs are estimated by within-node averages. For fixed partitions
and decisions, these regret averages converge almost surely
at positive-probability nodes under integrability.

\paragraph{Hierarchical clustering formulation}
Motivated by \eqref{eq:surrogate-tree-bound}, we minimize total
empirical surrogate regret over sample groups and auxiliary
candidate decisions within a prescribed node budget:
\begin{align}\label{eq:regret_based_clustering}
    \min_{\{ S_i \}, \{ x_{t,i} \}} \quad &
    \frac{1}{N} \sum_{t=1}^T \sum_{i \in \mathcal N_t} \sum_{k \in S_i}
    \widetilde A_{t,\theta}(x_{t,i}; x_{t-1,\pa(i)},\hist{t}^{(k)}) \\
    \text{s.t.}\quad &
    x_{t,i} \in \widetilde X_t(x_{t-1,\pa(i)}), \qquad t \in \br{T},\ i \in \mathcal N_t, \nonumber\\
    &
    \bigcup_{j\in\ch(i)}S_j=S_i, \qquad
    S_j\cap S_\ell=\varnothing,
    \qquad j\neq\ell\in\ch(i). \nonumber
\end{align}
Here $S_0=\br{N}$ and $x_{0,0}=x_0$. 
The decisions in this construction problem are auxiliary
representatives that serve as provisional predictions of
decisions along the tree. They determine the incoming
conditions used to construct subsequent stages, but do not
fix the implemented policy. After constructing the tree,
we solve its stochastic program over the original feasible sets
$X_t$ to obtain the decisions used by the controller.

\paragraph{Stagewise refinement algorithm}
Jointly solving \eqref{eq:regret_based_clustering} is
computationally demanding because groups and auxiliary decisions
are coupled across stages. We therefore construct the tree
forward in time, following \textcite[\S4.2]{heitsch2009multistage}, using two simplifications.
First, groupings and auxiliary decisions from earlier stages remain fixed.
Second, refinement considers only the current stage's surrogate
regret, omitting the effect of current choices on later groupings.
These simplifications reduce the computational burden by replacing
the joint hierarchical clustering problem with a sequence of
smaller local clustering problems.
At each stage, we allocate children among parents and, for
prescribed child counts, refine each parent's sample groups
and auxiliary decisions independently.

To guide the allocation, we evaluate the current child partition
of each parent $i\in\mathcal N_{t-1}$ through its summed
surrogate advantages:
\begin{equation}
\begin{aligned}
r_{t,i}\left(
\{(S_j,x_{t,j})\}_{j\in\ch(i)};x_{t-1,i}
\right)
\defeq\sum_{j\in\ch(i)}\sum_{k\in S_j}
\widetilde A_{t,\theta}(x_{t,j};x_{t-1,i},\hist{t}^{(k)}).
\end{aligned}
\label{eq:parent-regret}
\end{equation}
We write $r_{t,i}$ when the arguments are clear from the context.
Every parent starts with one child containing all its samples.
We then allocate up to $m_t$ additional children across the entire stage,
so the stage size satisfies
\[
|\mathcal N_t|\leq|\mathcal N_{t-1}|+m_t.
\]
At every iteration, we select the eligible parent with largest $r_{t,i}$ and
recompute its child groups and decisions. The eligible set
$\mathcal E_t$ contains parents admitting a split.
The score $r_{t,i}$ measures the regret of the current grouping
and does not estimate the improvement from a proposed split.
Figure~\ref{fig:scenario_reduction} illustrates this procedure.

For the selected parent $i \in \mathcal{N}_{t-1}$, let $\widetilde S_i$ and
$\widetilde x_{t-1,i}$ denote its fixed sample group and incoming
decision. 
After adding
one child, we recompute all child groups and representative
decisions by approximately solving the local regret-minimization
problem
\begin{align}
\min_{\{(S_j,x_{t,j})\}_{j\in\ch(i)}}\quad&
\sum_{j\in\ch(i)}\sum_{k\in S_j}
\widetilde A_{t,\theta}(x_{t,j};\widetilde x_{t-1,i},\hist{t}^{(k)})
\label{eq:regret_clustering}\\
\text{s.t.}\quad&
x_{t,j}\in\widetilde X_t(\widetilde x_{t-1,i}),
\qquad j\in\ch(i),
\nonumber\\
&
\bigcup_{j\in\ch(i)}S_j=\widetilde S_i, \qquad
S_j\cap S_\ell=\varnothing,
\qquad j\neq\ell\in\ch(i). \nonumber
\end{align}
The local problem resembles a constrained clustering problem, with feasible
decisions serving as representatives and surrogate advantage
measuring assignment cost. We therefore use a Lloyd-style
scheme to solve this problem by alternating between assignments and
representative updates. Starting from a few proposed shared decisions,
we assign each sampled history to the proposal with the lowest predicted
excess cost, then choose for each resulting group the shared decision with
the lowest total predicted excess cost. Write the current child groups
and decisions as $\widetilde S_j$ and $\widetilde x_{t,j}$.
The parent values $\widetilde S_i$ and
$\widetilde x_{t-1,i}$ remain fixed throughout:
\begin{itemize}[leftmargin=*,itemsep=0pt]
\item \textbf{Assignment step.} With representative decisions \( \{ \widetilde x_{t,j} \}_{j\in\ch(i)} \)
fixed, assign each sample to the child with lowest surrogate advantage:
\begin{equation}
\widetilde S_j\gets\left\{k\in\widetilde S_i\wmiddle
j=\argmin_{\ell\in\ch(i)}
\widetilde A_{t,\theta}(\widetilde x_{t,\ell};\widetilde x_{t-1,i},\hist{t}^{(k)})\right\}.
\label{eq:regret-assignment-step}
\end{equation}
\item \textbf{Update step.} With groups \( \{ \widetilde S_j \}_{j\in\ch(i)} \) fixed, choose a feasible
candidate representative minimizing each group's summed surrogate advantages:
\begin{equation}
\widetilde x_{t,j}\gets\argmin_{x^+\in\widetilde X_t(\widetilde x_{t-1,i})}
\sum_{k\in\widetilde S_j}\widetilde A_{t,\theta}(x^+;\widetilde x_{t-1,i},\hist{t}^{(k)}).
\label{eq:regret-update-step}
\end{equation}
\end{itemize}
Efficient construction requires surrogate representations that
make the update in \eqref{eq:regret-update-step} easy to compute.
The case studies use DQN and NAF representations for this purpose.
With exact representative updates,
accepted steps cannot increase local regret and continue until
assignments stabilize or a stopping criterion is met. 
This ensures that refinement cannot
increase the total surrogate regret accumulated through the current stage. Because earlier partitions and decisions are not revisited, the resulting tree is not guaranteed to be globally optimal.
Algorithm~\ref{alg:problem_based_reduction} summarizes the construction procedure.

For shrinking-horizon problems, every decision-rooted subtree
must additionally obey the growth schedule relative to
its own root, so the previous continuation remains an
admissible candidate. Keeping the previous
policy whenever a new one has a higher expected remaining
cost ensures that reoptimization cannot increase expected
cost relative to the initial tree policy \parencite[\S2.4]{bertsekas2020rollout}.

\begin{algorithm}[tbp]
\caption{Regret-Based Tree Construction}\label{alg:problem_based_reduction}
\KwIn{Sample set $\{ \omega_k \}_{k=1}^N$, maximum number of additional nodes per stage $m_t$, initial decision $x_0$}
\KwOut{Scenario tree $\T$ with corresponding $\{\widetilde S_{i} \}_{i \in \mathcal{N}_t,\,t \in \br{T}}$}

$\widetilde{x}_{0,0} \gets x_0$, $\mathcal{N}_0 \gets \{0 \}, \widetilde S_{0} \gets \{ 1, \dots, N \}$\\
\For{$t=1$ \KwTo $T$}{
    \For{ $i\in \mathcal N_{t-1}$}{
        1. Initialize $\ch(i) = \{ j \}$ with one node and $\widetilde S_{j} = \widetilde S_{i}$\\
        2. $\widetilde{x}_{t,j} \gets \text{argmin}_{x^+ \in \widetilde X_t(\widetilde{x}_{t-1,i})} \sum_{k \in \widetilde S_{j}} \widetilde A_{t,\theta}(x^+; \widetilde{x}_{t-1,i},\hist{t}^{(k)})$ \\
        3. $r_{t,i}\gets
        \sum_{j\in\ch(i)}\sum_{k\in\widetilde S_j}
        \widetilde A_{t,\theta}
        (\widetilde x_{t,j};\widetilde x_{t-1,i},\hist{t}^{(k)})$\\
    }
    \For{$l = 1 \dots m_t$}{
          Update $\mathcal E_t$ using the partition and budget constraints \\
          \If{$\mathcal E_t=\varnothing$}{
              \textbf{break} \\
          }
          1. Pick node $i^\star\in\text{argmax}_{i\in\mathcal E_t}\; r_{t,i}$\\
          2. Insert new node $\ch(i^\star) \gets \ch(i^\star) \cup \{ j \}$ \\
          3. Solve \eqref{eq:regret_clustering} to update $\{\widetilde S_{j}\}_{j \in \ch(i^\star)}$, $\{\widetilde{x}_{t,j}\}_{j \in \ch(i^\star)}$, $r_{t,i^\star}$
    }
    $\mathcal N_t\gets
    \bigcup_{i\in\mathcal N_{t-1}}\ch(i)$\
}
\end{algorithm}



\begin{figure}
    \centering
    \includegraphics[width=\textwidth]{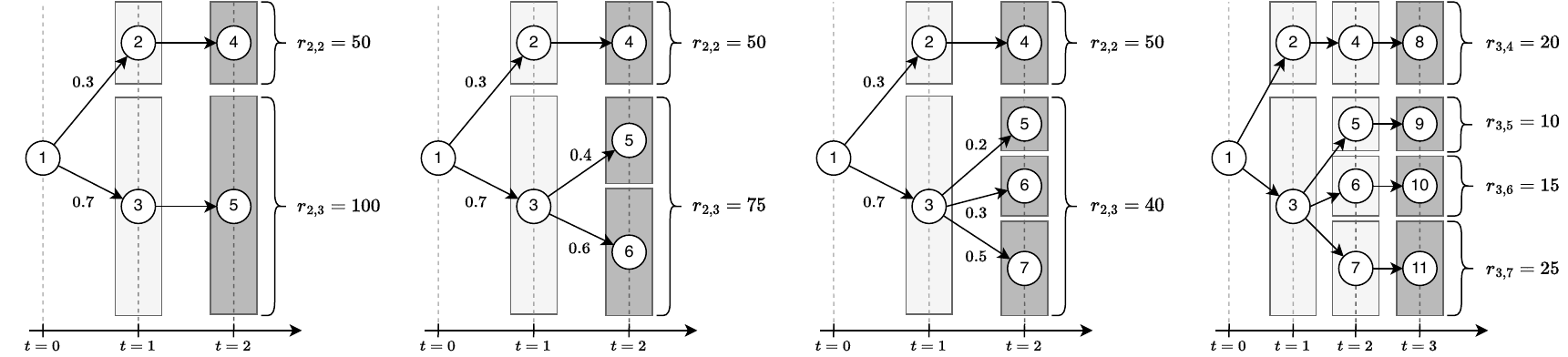}
    \caption{Stagewise refinement with fixed parent groups. Each regret-based refinement adds a child to the eligible parent with the largest current regret, then updates its child groups and representative decisions. Up to $m_t$ additional children are allocated across the stage before construction advances to the next stage.}\label{fig:scenario_reduction}
\end{figure}

\FloatBarrier

\section{Case Studies}\label{sec:results}
We now evaluate the proposed method on a stochastic energy storage problem with synthetic and historical
Belgian imbalance prices, and on a stochastic production planning problem with multidimensional
correlated uncertainty.

\paragraph{Shrinking-horizon implementation}
At each time step, we construct a scenario tree over the remaining
horizon, conditional on the observed history and the preceding decision.
We then solve the resulting stochastic program and implement only its
first-stage decision. After the next observation, we repeat the procedure
with the updated information. Rebuilding the tree allows its limited
budget to be focused on the uncertainty relevant to the current
conditions. Future decisions remain provisional and are revised before
implementation. 

\paragraph{Closed-loop evaluation and metrics}
The tree-based controllers follow the shrinking-horizon
stochastic model predictive control (SMPC) procedure
described above. Since the current observation is known,
each tree has a single first-stage decision node, with
branching beginning at the second stage.

In the experiments, we report mean episode reward (negative cumulative cost), with
$\pm$ values and shaded bands indicating one standard error
across evaluation episodes. For tree-based controllers, we also
report predicted total episode cost, defined at each reoptimization
as the cost already incurred plus the optimized expected remaining
cost $\hat v$. These predictions are averaged over decision times
and episodes.
For the same state, remaining horizon, and uncertainty model,
a lower $\hat v$ indicates a smaller optimality gap and hence
a better scenario-tree approximation.
However, earlier decisions affect both the costs incurred and
the states reached, so comparisons of the reported averages
across controllers do not isolate tree approximation quality.

\paragraph{Surrogate representations and training}
We use two surrogate architectures that support efficient
representative updates. A deep Q-network
(DQN)~\parencite{mnih2015human} provides action-value
estimates over a finite feasible action set, allowing
updates by enumeration. A normalized advantage function
(NAF)~\parencite{gu2016continuous} represents the action
value as a convex quadratic function of the control.

Both architectures are trained through exploratory
simulation using the control formulation of
Example~\ref{ex:optimal-control}. Their inputs contain
the physical state and predictive information, and their
outputs estimate the cost of a chosen control followed
by optimal continuation. Given the incoming state, the
control and system dynamics determine the corresponding
augmented decision \(x_t\), allowing these estimates to
serve as the decision-value surrogates used in tree
construction. The feasible augmented decisions represented
by each architecture form its candidate set $\widetilde X_t(x)$.

After hyperparameter tuning, both networks use one hidden
layer with 256 ReLU units. Training takes approximately
two hours for 100,000 episodes on the specified hardware,
while network evaluation takes only a few milliseconds.
During refinement of a stage, the sampled histories and incoming
representative decisions remain fixed, so the network outputs
can be cached and reused. A complete construction pass therefore
requires at most $NT$ network evaluations.

The scenario-tree optimization problems are solved using
Clarabel~\parencite{paul2024clarabel}.
All experiments are run on a single core of a Mac Studio
with an Apple M3 Ultra chip.
Implementation details and software dependencies are
provided in the accompanying code repository.

\paragraph{Tree sizes and benchmark methods}
Unless stated otherwise, scenario trees are constructed from $N=30{,}000$ Monte Carlo samples, and closed-loop performance is evaluated on $M=1000$ evaluation episodes of length $T=40$. In these experiments, every benchmark uses the same evaluation episodes and the samples are also used to estimate
node probabilities by relative frequencies and conditional stage costs by within-node averages.
To assess sensitivity to the node budget, we use a budget cap $B \in \{ 2, 4, 8, 16, 32 \}$ and define the stagewise growth schedule as
\begin{equation}
    m_1=0, \qquad
    m_t=\min(2^{t-1},B), \quad t\ge2.
\end{equation}
Here $m_t$ caps the number of additional nodes allocated across stage $t$
in Algorithm~\ref{alg:problem_based_reduction}, after assigning one child
to every parent. Thus $B$ limits the number of these additional nodes
per stage, rather than the total tree size or the number of children
per parent.

We compare the proposed regret-based tree construction method against four baselines:
\begin{enumerate}
\item \textbf{sampling-based construction}, which randomly selects additional scenario representatives according to the stagewise budget and assigns the remaining samples to their nearest selected representative.
\item \textbf{distance-based construction}, which uses the stagewise refinement structure with squared Euclidean distance instead of surrogate regret. This construction method minimizes the 2-Wasserstein distance between the original and reduced scenario distributions.
\item \textbf{standalone RL policies}, which select decisions directly from the regret surrogates, without constructing scenario trees or solving multistage stochastic programs.
\item \textbf{deterministic MPC}, which replaces the scenario tree with a single nominal trajectory defined by the stagewise expected uncertainty.
\end{enumerate}
For the sampling- and distance-based methods, Euclidean distances
are computed after rescaling each price coordinate to a common
variance, preventing differences in scale from dominating
the distance measure.
For synthetic storage, a numerical dynamic-programming (DP)
reference also supplies decision values for regret-based
construction. It uses uniform inventory and nominal-power grids,
sampled prices to approximate expectations, and linear interpolation
of continuation values. The standalone DP policy selects actions
directly from these approximate values, whereas DP-guided SMPC uses them
to construct a tree and then solves its stochastic program.

\subsection{Stochastic Energy Storage Optimization}
We consider a stochastic energy storage optimization problem under price uncertainty, a canonical example of multistage stochastic optimization \parencite{bruneel2025multistage}. At each time step, the decision-maker observes the current electricity price and decides whether to charge, remain idle, or discharge to exploit future price fluctuations while respecting the operational constraints of the storage system.

\subsubsection{Problem Description}
At each shrinking-horizon solve, $I_0$ denotes the current stored energy and $t_{\mathrm{rem}}\in\br{T}$ the number of remaining stages.
The problem is
\begin{equation}
    \begin{aligned}
        v^\star
        \defeq
        \inf_{\{u_t^+,u_t^-\}_{t=1}^{t_{\mathrm{rem}}}}
        \quad &
        \mathbb{E}\left[
        \sum_{t=1}^{t_{\mathrm{rem}}}
        \xi_t \left(u_t^+ - u_t^-\right) \Delta t
        + \mu \left(u_t^+ + u_t^-\right) \Delta t
        \right]  \\
        \text{s.t. }
        & I_t = I_{t-1} + \eta u_t^+ \Delta t - \tfrac{1}{\eta}u_t^- \Delta t,
         \\
        & 0 \le I_t \le I_\mathrm{max}, \\
        & I_{t_{\mathrm{rem}}} = 0.5 I_\mathrm{max}, \\
        & u_t^+ \in [0,q],
        \qquad u_t^- \in [0,q],
         \\
        & u_t^+,u_t^- \text{ are } \mathcal{F}_t\text{-measurable},
        \qquad t=1,\ldots,t_{\mathrm{rem}}.
    \end{aligned}
\end{equation}
Here, $\xi_t$ is the uncertain electricity price, $I_t$ the stored energy, and $u_t^+$ and $u_t^-$ the charging and discharging powers.
We use $\Delta t = 15$ min, capacity $I_\mathrm{max}=1$ MWh, power limit $q=0.5$ MW, efficiency $\eta=0.9$, and cycle cost $\mu=125$ EUR/MWh.
Episodes start at $50\%$ capacity and simultaneous charging and discharging are permitted to retain a linear programming formulation.

This problem fits the MSP formulation \eqref{eq:full-problem} by taking the augmented stage decision $x_t=(I_t,u_t^+,u_t^-)$ and including the inventory dynamics and operational constraints in $X_t(x_{t-1})$. Since costs are linear
in price, the tree problem~\eqref{eq:discretized-problem}
uses conditional mean prices at its nodes.

\subsubsection{Storage Optimization with Synthetic Price Data}\label{sec:synthetic_storage}
The synthetic prices are assumed to be stagewise independent and normally distributed with mean $\mu_\xi=100$~EUR/MWh and standard
deviation $\sigma_\xi=90$~EUR/MWh.
Both surrogates use $(I_{t-1}, \xi_t, t_\mathrm{rem})$ as input.
DQN uses three nominal actions corresponding to charge $q$,
discharge $q$, or hold, while NAF uses continuous requests in
$[-q,q]$. Each request is converted into a proposed next inventory,
clipped to satisfy capacity limits and preserve terminal reachability,
and converted back into feasible power. Thus, even a full-power
or hold request can produce a partial or corrective action.
The corresponding augmented decisions form the candidate set
$\widetilde X_t(x)$.

Figure~\ref{fig:dqn_training} shows the validation reward during training
and compares the learned advantage functions with the numerical DP reference.
Figure~\ref{fig:clustering} illustrates how these estimates affect scenario aggregation. Cycle costs and conversion losses discourage trading on price fluctuations with small expected gains. Regret-based construction therefore places partition boundaries near the thresholds between holding, charging, and discharging. This is also reflected in the topology of the resulting scenario tree as shown in Figure~\ref{fig:scenario_tree_structure}.

Figure~\ref{fig:scenario_tree_sizes} shows that the regret-based controllers achieve higher mean rewards than the sampling- and distance-based baselines across the tested node budgets. DP-guided SMPC leads the tree-based methods,
while DQN-guided SMPC leads the learned variants.
Both learned variants improve on their standalone policies,
particularly NAF. Increasing the tree size improves mean
rewards across all tree-based methods.

Table~\ref{tab:tree_construction_times} shows that construction
and solution times increase approximately linearly with the
reported mean node count over the tested budget range.

\begin{figure}
    \centering
    \subfloat[Training reward]{\includegraphics[width=0.32\textwidth]{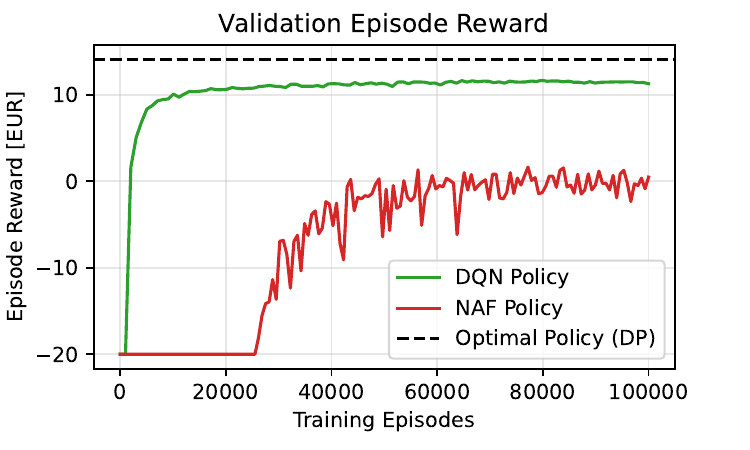}}\hfill
    \subfloat[DQN advantage function]{\includegraphics[width=0.32\textwidth]{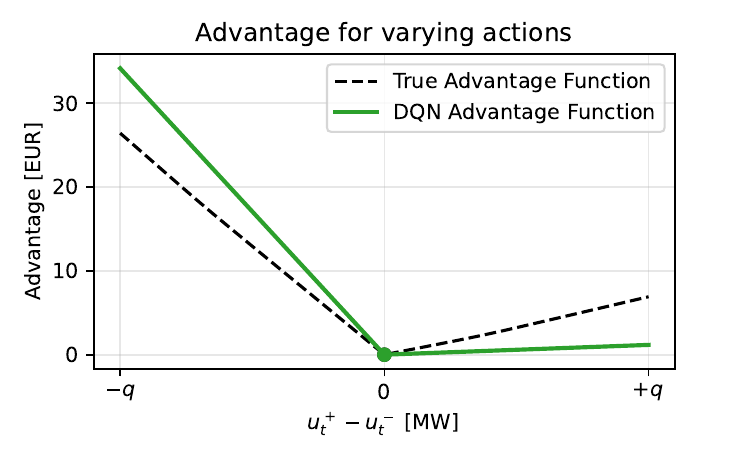}}\hfill
    \subfloat[NAF advantage function]{\includegraphics[width=0.32\textwidth]{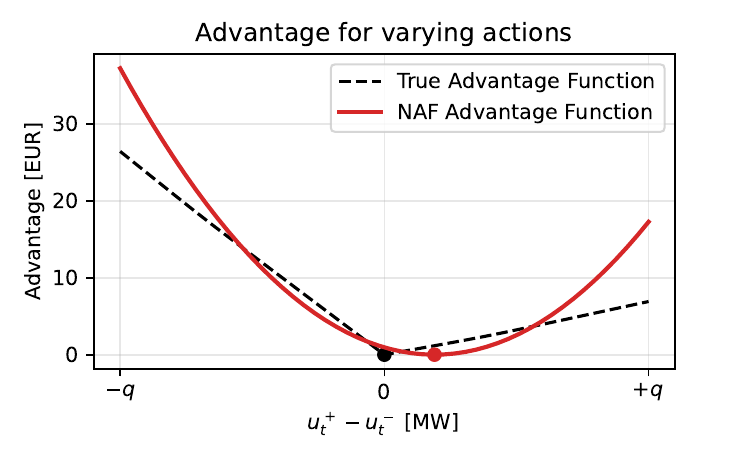}}

    \caption{Training reward and learned advantage functions on the
    synthetic storage problem. Left: validation reward during training.
    Center and right: learned advantage functions compared with the
    numerical DP reference at $t_\mathrm{rem}=20$,
    $I_{t-1}=0.65 I_\mathrm{max}$, and $\xi_t=50$~EUR/MWh.
    The DQN follows the piecewise-linear structure of the reference
    more closely than the quadratic NAF representation.}
    \label{fig:dqn_training}
\end{figure}

\begin{figure}
     \centering
     \includegraphics[width=\textwidth]{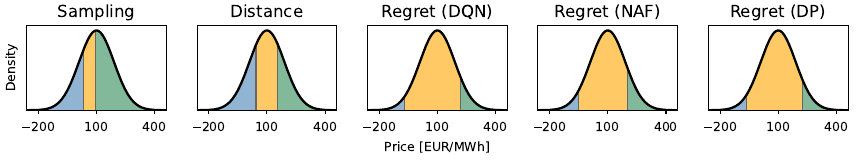}
     \caption{Price partitions induced by the history-to-node mapping
        $\kappa_2$ in the synthetic storage problem, with
        $t_\mathrm{rem}=20$, $I_{t-1}=0.65 I_\mathrm{max}$, in the first stage of uncertainty.
        Each region corresponds to prices assigned to the same node. The regret-based methods lead to very similar partitions.
     }\label{fig:clustering}
\end{figure}

\begin{figure}
    \centering
    \includegraphics[width=\textwidth]{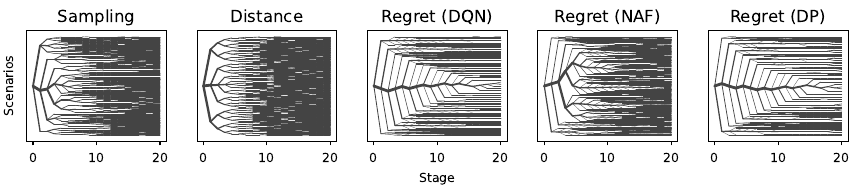}
    \caption{Scenario-tree topologies for $t_{\mathrm{rem}}=20$, $I_{t-1}=0.65 I_\mathrm{max}$ and $B=32$ in the synthetic storage problem of Section~\ref{sec:synthetic_storage}. The distance-based method produces a nearly balanced tree, whereas the regret-based methods yield more uneven branching patterns that reflect the decision relevance of different price regions. In particular, the DQN-based and DP-based trees contain a dominant branch associated with the hold decision, together with smaller branches representing extreme prices that trigger charging or discharging.}\label{fig:scenario_tree_structure}
\end{figure}


\begin{figure}
    \centering
    \subfloat[Episode Reward]{\includegraphics[width=0.5\textwidth]{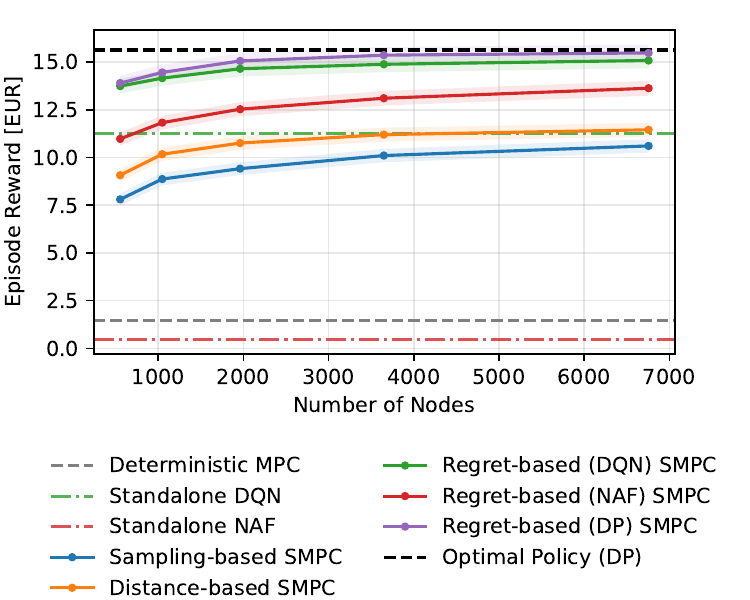}}
    \subfloat[Average Objective $\hat{v}$]{\includegraphics[width=0.5\textwidth]{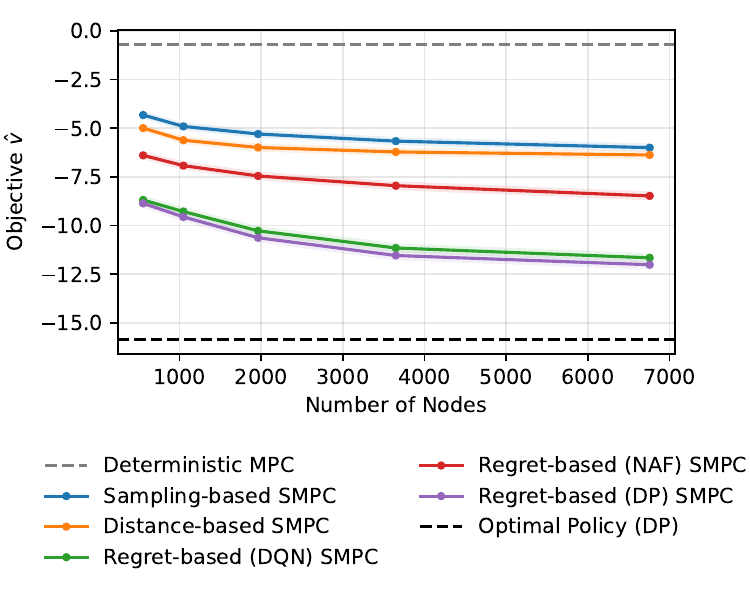}}

    \caption{Effect of scenario-tree size on the synthetic storage problem of Section~\ref{sec:synthetic_storage}:
    mean episode reward (left) and predicted total episode cost (right).
    Predicted costs combine costs already incurred with the optimized
    expected remaining cost and are averaged over decision times and
    episodes. Shaded bands indicate one standard error.}\label{fig:scenario_tree_sizes}
\end{figure}

\begin{table}
    \centering
    \begin{tabular}{@{}rrrrrrr@{}}
    \toprule
    Budget cap $B$ & Mean Nodes & $t_\mathrm{cl,sb}$ & $t_\mathrm{cl,db}$ & $t_\mathrm{cl,dqn}$ & $t_\mathrm{cl,naf}$ & $t_\mathrm{solve}$ \\
    \midrule
    2 & 554 & 0.26s & 0.31s & 0.48s & 0.65s & 0.02s \\
    4 & 1048 & 0.35s & 0.43s & 0.67s & 0.88s & 0.04s \\
    8 & 1961 & 0.53s & 0.66s & 0.97s & 1.30s & 0.08s \\
    16 & 3649 & 0.88s & 1.01s & 1.50s & 2.12s & 0.15s \\
    32 & 6757 & 1.75s & 1.85s & 2.44s & 3.97s & 0.31s \\
    \bottomrule
    \end{tabular}
    \caption{Average per-step computation times for different scenario tree sizes in the synthetic storage problem of Section~\ref{sec:synthetic_storage}. The cap $B$ bounds the additional children allocated across a stage after every parent receives one child. A larger budget cap leads to more nodes in the scenario tree and consequently higher computation times. The columns $t_\mathrm{cl,sb}$, $t_\mathrm{cl,db}$, $t_\mathrm{cl,dqn}$, and $t_\mathrm{cl,naf}$ report the average tree construction times for the sampling-based, distance-based, regret-based DQN, and regret-based NAF methods, respectively. The column $t_\mathrm{solve}$ reports the average time required to solve the resulting SMPC problem.
    }\label{tab:tree_construction_times}
\end{table}

\subsubsection{Storage Optimization with Real Price Data}\label{sec:real_storage}
We consider a storage setup similar to the synthetic case,
using historical Elia imbalance prices from 2023--2025
at 15-minute resolution.
To limit the influence of extreme
price spikes on average episode rewards, observations are clipped
at the 0.5th and 99.5th percentiles.
The dependence model and preprocessing
parameters are fitted using the full 2023--2025 period, and all
methods are evaluated on the same 5000 historical episodes
sampled from that period.

To generate uncertainty scenarios, prices are standardized using the rolling
mean and standard deviation over the preceding 90 days.
At each reoptimization, the method of \textcite{pinson2009probabilistic} is used to generate temporally correlated
standardized price paths conditional on the observed standardized
price history, using the fitted dependence model.
The resulting paths are transformed back to the original price
scale using the rolling mean and standard deviation.
The surrogates are trained on synthetic trajectories
generated by the fitted dependence model. Their inputs
are augmented with the rolling statistics and the eight
most recent historical prices. Evaluating the resulting controllers
on historical episodes provides a retrospective comparison
under a common estimated price model. 

Figure~\ref{fig:scenario_tree_sizes_real} shows results
broadly consistent with the synthetic experiments, with
the regret-based methods achieving the highest mean rewards.
However, the lower predicted episode costs of NAF-guided SMPC
relative to distance-based SMPC do not necessarily translate
into higher realized rewards.
One possible explanation is that the assumed time-invariant distribution of standardized prices does not fully capture changes in the underlying price dynamics.
More flexible price models could better capture
the changing dynamics \parencite{bruneel2025optimal},
but are outside the scope of this study.



\begin{figure}
    \centering
    \subfloat[Episode Reward]{\includegraphics[width=0.5\textwidth]{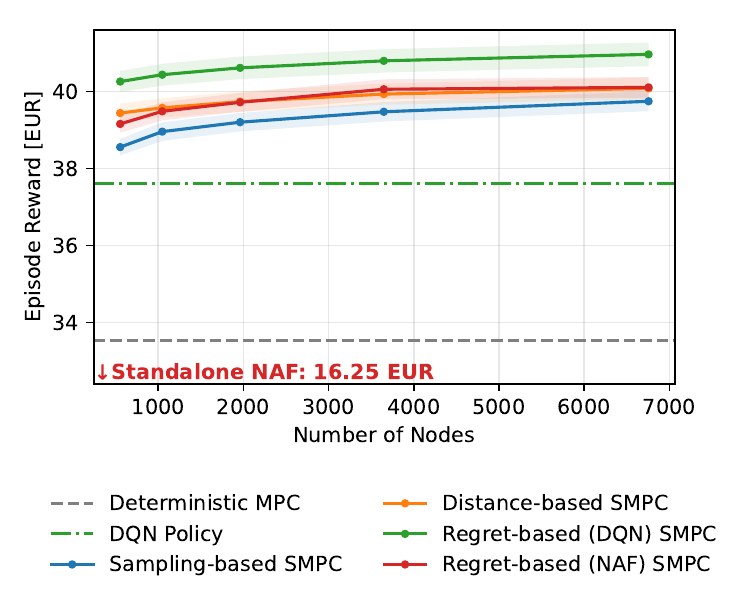}}
    \subfloat[Average Objective $\hat{v}$]{\includegraphics[width=0.5\textwidth]{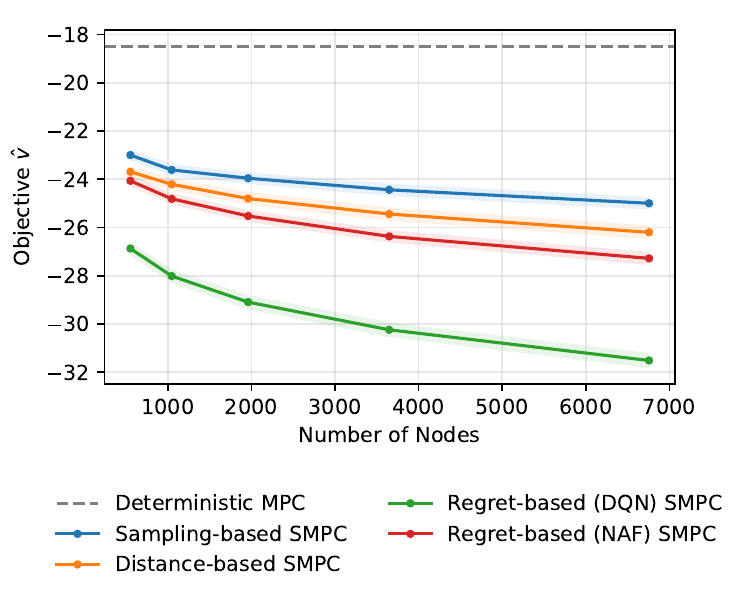}}

    \caption{Mean episode reward and predicted total episode cost
    for the energy storage example of Section~\ref{sec:real_storage} across scenario-tree sizes.
    Results are based on 5000 historical episodes from
    2023--2025, the period also used to fit the price model.
    Predicted costs are averaged over decision times and episodes.
    Shaded bands indicate one standard error.}
\label{fig:scenario_tree_sizes_real}
\end{figure}

\subsection{Stochastic Production Planning}\label{sec:stochastic_production}
Finally, we evaluate the proposed method on a stochastic production planning problem with multidimensional, correlated uncertainty. A finished product is assembled from $J=4$ components with stochastic input prices. Production requires a fixed lead time of $L=5$ time steps, and the finished product is sold at a stochastic price. The objective is to maximize expected profit over a horizon of $T=40$ stages.

This setting is a useful test bed for regret-based scenario-tree construction because the relevance of each uncertainty component is dependent on the current inventory levels. For example, when no finished products are in stock, the selling price has little immediate effect on the current decision. An efficient scenario tree should therefore avoid allocating excessive resolution to the selling price in such settings and instead focus on modeling the other prices.


\subsubsection{Problem Description}
Let $E_t\in\mathbb{R}^J_+$ denote component inventories, $W_t\in\mathbb{R}^L_+$ the work-in-progress pipeline, and $F_t\in\mathbb{R}_+$ the finished-goods inventory at the end of stage $t$. Given the incoming inventories, the controller chooses component purchases $u^{\mathrm{b}}_t\in\mathbb{R}^J_+$, production starts $u^{\mathrm{p}}_t\in\mathbb{R}_+$, and sales $u^{\mathrm{s}}_t\in\mathbb{R}_+$.
The uncertainty vector $\xi_t = (c_t, a_t, s_t)$ consists of component prices $c_t\in\mathbb{R}^J$, assembly cost $a_t\in\mathbb{R}$, and a residual product spread $s_t\in\mathbb{R}$. The finished-product selling price is
\[
    r_t = \mathbf{1}^\top c_t + a_t + s_t,
\]
so that the output price is strongly linked to the input costs while the spread $s_t$ captures the decision-relevant margin. The objective is to maximize the expected total profit over the remaining horizon $t_{\mathrm{rem}}\in\br{T}$, subject to inventory dynamics and operational constraints:
\begin{equation}
    \begin{aligned}
        v^\star
        \defeq
        \inf_{\{u^{\mathrm{b}}_t,u^{\mathrm{p}}_t,u^{\mathrm{s}}_t\}_{t=1}^{t_{\mathrm{rem}}}}
        \quad &
        \mathbb{E}\left[
        \sum_{t=1}^{t_{\mathrm{rem}}}
        c_t^\top u_t^{\mathrm{b}}
        + a_t\, u_t^{\mathrm{p}}
        - r_t\, u_t^{\mathrm{s}}
        \right]  \\
        \text{s.t.}\quad
        & E_{t,j} = E_{t-1,j} + u^{\mathrm{b}}_{t,j} - u^{\mathrm{p}}_t,
        \quad j = 1,\ldots,J,
         \\
        & W_{0,t} = u^{\mathrm{p}}_t,
         \\
        & W_{\ell,t} = W_{\ell-1,t-1},
        \quad \ell=1,\ldots,L-1,
         \\
        & F_{t} = F_{t-1} + W_{L-1,t-1} - u^{\mathrm{s}}_t,
         \\
        & (u^{\mathrm{b}}_t, u^{\mathrm{p}}_t, u^{\mathrm{s}}_t)
          \in X_t(E_{t-1}, W_{t-1}, F_{t-1}),
         \\
        & u^{\mathrm{b}}_t, u^{\mathrm{p}}_t, u^{\mathrm{s}}_t
          \text{ are } \mathcal{F}_t\text{-measurable},
        \quad t=1,\ldots,t_{\mathrm{rem}}.
    \end{aligned}
\end{equation}
All purchase, production, and sales quantities lie in
$[0,1]$, and inventories must remain nonnegative.
The total quantity in the production pipeline cannot
exceed one unit.
Episodes start with empty component inventories, an empty
production pipeline, and no finished products. Remaining
inventories have no terminal salvage value.

The uncertainty follows a mean-reverting vector autoregression,
\begin{equation}
    \xi_{t+1} = \mu + \Phi(\xi_t-\mu) + \varepsilon_{t+1},
    \qquad
    \varepsilon_{t+1}\sim\mathcal N(0,\Sigma).
\end{equation}
The matrix $\Phi$ is diagonal. For each component price,
the assembly cost and the product spread,
the (mean, AR coefficient, innovation standard deviation)
triples are $(1,0.20,0.12)$, $(5,0.25,1)$, and
$(-3,0.10,5)$, respectively. Distinct component innovations have
correlation $0.15$. All other innovation cross-correlations
are zero. Similar to the energy storage problem, the stage cost is linear in the uncertainty, so that the information-coarsened problem~\eqref{eq:discretized-problem} reduces to a linear program parametrized by the conditional mean of $\xi_t$ at each node.
An equivalent MSP representation uses the augmented stage decision
$x_t=(E_t,W_t,F_t,u_t^{\mathrm b},u_t^{\mathrm p},u_t^{\mathrm s})$.
The inventory and pipeline equations, together with the admissible-control constraints, define the corresponding mapping $X_t(x_{t-1})$.

\subsubsection{Results}
Both surrogate models use the input
$(E_{t-1}, W_{t-1}, F_{t-1}, \xi_t, t_\mathrm{rem})$.
DQN uses $2^6=64$ nominal action combinations, with each dimension
selecting either zero or its maximum request. The resulting candidates
are subject to the joint inventory and pipeline constraints.
NAF uses continuous controls subject to the same constraints.

Figure~\ref{fig:scenario_tree_sizes_assembly} summarizes
the results on 1000 evaluation episodes. The regret-based
controllers achieve higher mean rewards than the sampling-
and distance-based baselines across the tested node budgets.
DQN-guided SMPC leads the tree-based methods and exceeds
standalone DQN at larger budgets. Distance-based construction
performs worse than sampling-based construction despite
variance normalization.

At $B=32$, mean computation times per reoptimization,
including scenario generation, tree construction, and
optimization, are 6.32, 6.54, 6.42, and 7.04~s for
sampling-based, distance-based, DQN-guided, and NAF-guided
SMPC, respectively. Deterministic MPC requires 6~ms,
while standalone DQN and NAF evaluations require
0.163 and 0.201~ms, respectively.


\begin{figure}
    \centering
    \subfloat[Episode Reward]{\includegraphics[width=0.5\textwidth]{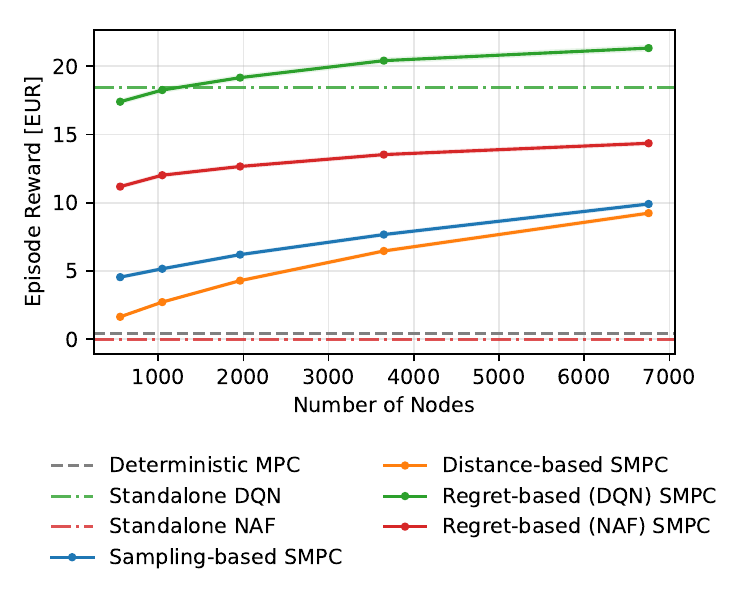}}
    \subfloat[Average Objective $\hat{v}$]{\includegraphics[width=0.5\textwidth]{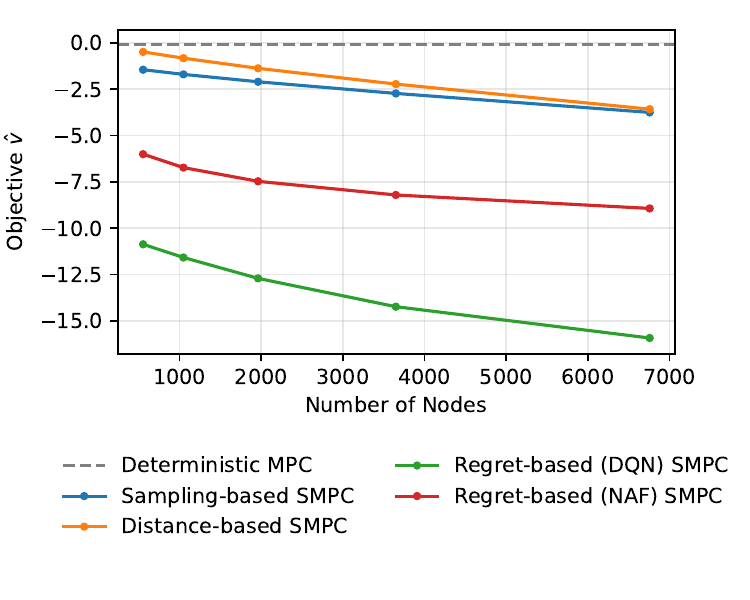}}

    \caption{Mean episode reward and predicted total episode cost
    across scenario-tree sizes for the production-planning problem of Section~\ref{sec:stochastic_production}
    on 1000 evaluation episodes. Predicted costs are averaged over
    decision times and episodes. Shaded bands indicate one standard error.}\label{fig:scenario_tree_sizes_assembly}
\end{figure}

\section{Discussion}

\paragraph{Key findings}
The experiments support allocating tree resolution according to the cost of sharing decisions.
Histories with
substantially different uncertainty realizations can share
a node when a common decision incurs little regret, allowing
more nodes to be allocated where aggregation causes larger
decision losses.
In practice,
larger trees should be justified by their closed-loop gains
relative to the additional construction and solution time.

The surrogate serves a different purpose in tree construction
than in standalone action selection. During construction, they guide aggregation and branch allocation,
while the resulting tree problem determines the
implemented decisions. Poor standalone action
selection therefore need not imply poor construction quality.

\paragraph{Limitations}
The proposed approach targets information-coarsening loss
under an assumed probability model. Increasing tree
resolution does not correct misspecification of that model,
so a lower predicted episode cost need not translate into
better realized closed-loop performance. Distributionally
robust extensions with cost-aware ambiguity
sets~\parencite{schuurmans2023costaware} could mitigate
the effects of errors in the estimated distribution.

The stagewise construction and growth schedule $m_t$ remain
heuristic.
Early partitions and representative
decisions determine the incoming conditions used to construct
later stages, which can lead to suboptimal tree structures if early decisions are not well-informed.
Adaptive budget allocation and backward refinement could
improve construction by revisiting these choices.
Extending the construction to
stochastic constraints additionally requires a tractable way to ensure
feasibility for almost every history assigned to each node.

\section{Conclusion}
This paper showed that, under a fixed probability measure and deterministic feasibility constraints, the optimality gap caused by scenario-tree information coarsening decomposes exactly into a sum of nodewise regret terms, each quantifying the local decision loss. This decomposition provides a principled target for constructing a scenario tree: minimize the total regret caused by grouping histories rather than relying solely on statistical distance.

Building on this result, we developed a two-phase framework that learns surrogate regret functions offline via reinforcement learning and uses them online to guide stagewise tree construction through regret-based clustering. The resulting algorithm generalizes classical distance-based methods by replacing geometric proximity with estimated decision loss.

Experiments on storage systems and production planning demonstrate benefits over the considered statistical construction methods under the tested node budgets. The gains depend on the surrogate, tree size, and uncertainty model, and can persist even when the corresponding standalone learned policy performs poorly. These findings support using learned decision information to guide scenario-tree construction as a complement to multistage optimization.

\section*{Acknowledgements}
This work was supported by Flanders Innovation \&
Entrepreneurship (VLAIO) under Baekeland mandate
HBC.2022.0695
and the KU Leuven Research Council under grant C14/24/103.
The work of M. Schuurmans was supported by an FWO postdoctoral fellowship under project number 1248626N.

\section*{CRediT authorship contribution statement}
\textbf{Robin Bruneel:}
Writing -- original draft, Validation, Software,
Methodology, Investigation, Formal analysis, Data curation,
Conceptualization, Visualization, Funding acquisition.
\textbf{Mathijs Schuurmans:}
Writing -- review \& editing, Conceptualization, Methodology.
\textbf{Panagiotis Patrinos:}
Supervision, Conceptualization, Methodology.

\section*{Declaration of competing interest}
Robin Bruneel reports financial support was provided by Amplifino. If there are other authors, they declare that they have no known competing financial interests or personal relationships that could have appeared to influence the work reported in this paper.

\section*{Declaration of generative AI and AI-assisted
technologies in the manuscript preparation process}
During the preparation of this work, the authors used ChatGPT in order to assist with writing. After using this tool/service, the authors reviewed and edited the content as needed and take full responsibility for the content of the published article.

\section*{Data and code availability}
The historical Belgian imbalance prices used in this study
are publicly available from Elia at
\url{https://opendata.elia.be}.
The code supporting this study is available at
\url{https://github.com/BruneelRobin/regret-based-scenario-tree-construction}.

\printbibliography[title={References}]

\end{document}